\documentclass{amsart}

\usepackage[algo2e,lined,boxruled]{algorithm2e}
\usepackage{amssymb}
\usepackage{xcolor}
\usepackage{subcaption}
\usepackage{amsmath}
\usepackage{array}
\usepackage{tikz}\usetikzlibrary{shapes.geometric,backgrounds,patterns,decorations.pathmorphing, decorations.pathreplacing, decorations.shapes}
\usepackage{caption}
\usepackage[colorlinks=true, linkcolor=blue, citecolor=blue]{hyperref}
\usepackage[english]{babel}
\usepackage[autostyle]{csquotes}
\usepackage{float}
\usepackage{mathtools}
\usepackage{enumitem}
\usepackage{todonotes}

\usepackage[T1]{fontenc}

\pgfdeclarelayer{background layer}
\pgfdeclarelayer{foreground layer}
\pgfsetlayers{background layer,main,foreground layer}

\tikzset{clause3/.pic={
           \newdimen\R
            \R=2cm
            \foreach \x in {0,60,120,...,300}
                \node[circle, draw, fill=black, inner sep=2pt](\x) at (\x:\R) {};
            \begin{pgfonlayer}{background layer}
                \draw (0:2) \foreach \x in {60,120,...,360} {-- (\x:\R) }; 
                \draw (0:2) -- (180:2);
            \end{pgfonlayer}
             \node[anchor=north] at (-60:2) {};
             \node[anchor=south] at (180:2) {};
             \node[anchor=west] at (60:2) {};
            \node[kite,draw,pattern=grid, line width=1pt,minimum size =0.2cm, anchor=south,rotate=235] at (0:2){} ;
	        \node[kite,draw,pattern=grid, line width=1pt,minimum size =0.2cm, anchor=south,rotate=0] at (120:2){} ;
            \node[kite,draw,pattern=grid, line width=1pt,minimum size =0.2cm, anchor=south,rotate=115] at (240:2){} ;}
            }
\tikzset{clause1/.pic={
           \newdimen\R
            \R=2cm
            \foreach \x in {0,60,120,...,300}
                \node[circle, draw, fill=black, inner sep=2pt](\x) at (\x:\R) {};
            \begin{pgfonlayer}{background layer}
                \draw (0:2) \foreach \x in {60,120,...,360} {-- (\x:\R) }; 
                \draw (0:2) -- (180:2);
            \end{pgfonlayer}
              \node[anchor=south] at (-60:2) {};
             \node[anchor=west] at (180:2) {};
             \node[anchor=east] at (60:2) {};
            \node[kite,draw,pattern=grid, line width=1pt,minimum size =0.2cm, anchor=south,rotate=235] at (0:2){} ;
	        \node[kite,draw,pattern=grid, line width=1pt,minimum size =0.2cm, anchor=south,rotate=0] at (120:2){} ;
            \node[kite,draw,pattern=grid, line width=1pt,minimum size =0.2cm, anchor=south,rotate=115] at (240:2){} ;}
            }
\tikzset{clause2/.pic={
           \newdimen\R
            \R=2cm
            \foreach \x in {0,60,120,...,300}
                \node[circle, draw, fill=black, inner sep=2pt](\x) at (\x:\R) {};
            \begin{pgfonlayer}{background layer}
                \draw (0:2) \foreach \x in {60,120,...,360} {-- (\x:\R) }; 
                \draw (0:2) -- (180:2);
            \end{pgfonlayer}
             \node[anchor=south] at (-60:2) {};
             \node[anchor=west] at (180:2) {};
             \node[anchor=east] at (60:2) {};
            \node[kite,draw,pattern=grid, line width=1pt,minimum size =0.2cm, anchor=south,rotate=235] at (0:2){} ;
	        \node[kite,draw,pattern=grid, line width=1pt,minimum size =0.2cm, anchor=south,rotate=0] at (120:2){} ;
            \node[kite,draw,pattern=grid, line width=1pt,minimum size =0.2cm, anchor=south,rotate=115] at (240:2){} ;}
            }

\makeatletter
\def\ps@pprintTitle{%
 \let\@oddhead\@empty
 \let\@evenhead\@empty
 \def\@oddfoot{\centerline{\thepage}}%
 \let\@evenfoot\@oddfoot}
\makeatother

\definecolor{Blue}{rgb}{0,0.5,0.5}

 \newtheorem{theorem}{Theorem}[section]
 \newtheorem{lemma}[theorem]{Lemma}
\newtheorem{question}[section]{Question}
\newcounter{prfcnt}
\newenvironment{proofcnt}{\setcounter{prfcnt}{1}}{}
\newcommand{\proofcount}[1]{{\itshape (\alph{prfcnt})} #1\addtocounter{prfcnt}{1} \newline}

\newcommand{\iC}{\mathit{C}}

\begin{document}

%%%%%%%%%%%%%%%%%%%%%%%%%%%%%%%%%%%%%%%%%%%%%%%%%%%%%%%%%%%%%%%%%%%%%%%%%%%%%%%%%%%%%%
% H E A D I N G
% \begin{frontmatter}
\title{Complexity of Near-3-Choosability Problem}
\author{Sounaka Mishra \and Rohini S \and Sagar S. Sawant }
\address{Department of Mathematics, Indian Institute of Technology Madras, Chennai 600 036, India.}
\email{sounak@iitm.ac.in, s.rohini@smail.iitm.ac.in, sagar@smail.iitm.ac.in}

% \ead{sounak@iitm.ac.in}
% \author[a1]{Rohini S}
% \ead{s.rohini@smail.iitm.ac.in}
% \author[a1]{Sagar S. Sawant}
% \ead{sagar@smail.iitm.ac.in}

% \address[a1]{Department of Mathematics, Indian Institute of Technology Madras, Chennai 600 036, India.}

% \cortext[cor1]{Corresponding author}
% \date{}
\maketitle
\begin{abstract}
It is currently an unsolved problem to determine whether a $\triangle$-free planar graph $G$ contains an independent set $A$ such that $G[V_G\setminus A]$ is $2$-choosable. However, in this paper, we take a slightly different approach by relaxing the planarity condition. We prove the $\mathbb{NP}$-completeness of the above decision problem when the graph is $\triangle$-free, $4$-colorable, and of diameter $3$. Building upon this notion, we examine the computational complexity of two optimization problems: minimum near $3$-choosability and minimum $2$-choosable deletion. In the former problem, the goal is to find an independent set $A$ of minimum size in a given graph $G$, such that the induced subgraph $G[V_G \setminus A]$ is $2$-choosable. We establish that this problem is $\mathbb{NP}$-hard to approximate within a factor of $|V_G|^{1-\epsilon}$ for any $\epsilon > 0$, even for planar bipartite graphs.
On the other hand, the problem of minimum $2$-choosable deletion involves determining a vertex set $A \subseteq V_G$ of minimum cardinality such that the induced subgraph $G[V_G \setminus A]$ is $2$-choosable. We prove that this problem is $\mathbb{NP}$-complete, but can be approximated within a factor of $O(\log |V_G|)$.
\end{abstract}
\keywords{Graph Coloring, List Coloring, $\mathbb{NP}$-complete, Approximation Algorithm}

\section{Introduction}
In this paper, we consider graphs that are finite, undirected, and simple. For a graph $G$, we denote $V_G$ and $E_G$ as the vertex set and edge set of $G$, respectively. We assume that $|V_G| =n$ and $|E_G|=m$.

Given a positive integer $\ell$ and a graph $G$, a \emph{proper $\ell$-coloring} of $G$ is a function $f:V_G \rightarrow \{1, 2, \cdots, \ell\}$ that assigns a color from $\{1, 2, \cdots, \ell\}$ to each vertex of $G$ such that for every edge $(u, v) \in E_G$, $f(u) \neq f(v)$. If a graph $G$ admits a proper $\ell$-coloring then we say that $G$ is \emph{$\ell$-colorable}. The chromatic number $\chi(G)$ of a graph $G$ is the smallest non-negative integer $\ell$ such that $G$ is properly $\ell$-colorable. A $k$-list assignment for a graph $G$ is a map $L$ that assigns each vertex $v \in V_G$ a set of permissible colors $L(v) \subset \mathbb{N}$ with $|L(v)| = k$. An $L$-coloring  of a graph $G$ is a proper coloring of $G$ that assigns each vertex $v \in V_G$ a color from its list $L(v)$. We say that $G$ is \emph{$L$-colorable} if there exists an $L$-coloring of $G$. For a positive integer $k$, we say that $G$ is \emph{$k$-choosable} if $G$ has an $L$-coloring, for every $k$-list assignment $L$. The \emph{list-chromatic number} (or choice number) $ch(G)$  of $G$ is the minimum positive integer $k$ such that $G$ is $k$-choosable. List coloring of a graph, which is a generalization of proper coloring, was introduced in \cite{erdos}. It has observed that if a graph is $k$-choosable, it is $\ell$-colorable for all $\ell \geq k$, but the converse is false. 

The $k$-choosability problem for a given graph $G$ involves deciding whether $G$ is a $k$-choosable graph. Previous research has shown that this problem is polynomial-time solvable when $k$ is less than or equal to 2, but becomes $\Pi_2^p$-complete (where $\mathbb{NP} \cup$ co-$\mathbb{NP} \subseteq \Pi_{2}^p$) for $k\geq 3$ \cite{erdos}.

A refinement of the list coloring with respect to the restriction of the list was given by Choi and Kwon in \cite{tcommon}. They considered the list coloring problem for $k$-list assignment having at least $t$ colors common. In particular, a  $k$-common $k$-list coloring is equivalent to the $k$-colorability. In \cite{Zhu}, Xuding introduced another refined scale of choosability. Let $\Lambda =\{\lambda_1,\lambda_2,\dots,\lambda_r\}$ be a partition of a positive integer $k$. A graph $G$ is said to be $\Lambda$-choosable if there exists a decomposition of $V_G$ into $(V_1,V_2,\dots,V_r)$ such that each induced subgraph $G[V_i]$ is $\lambda_i$-choosable for $1 \leq i \leq r$. Using $\Lambda$-choosability, the notion of $k$-colorability (equivalent to $\{1,1,\dots,1\}$-choosability) and $k$-choosabilty come under the same frame structure. Moreover, the $t$-common $k$-list coloring coincides with $\{1,1,\dots,1,k-t\}$-coloring where the part 1 occurs $t$ times. A partition $\Lambda$ of $k$ is a refinement of  another partition $\Gamma$ of $k$ if $\Lambda$ is obtained from $\Gamma$ by subdividing some parts of $\Gamma$. 
In fact, Xuding \cite{Zhu} proved that if a partition $\Lambda$ is a refinement of a partition $\Gamma$, then there exists a graph $G$ that is $\Lambda$-choosable but not $\Gamma$-choosable. 
This raises the following natural question:
\begin{question}{\label{ques: l-k_choosable}}
Let $k,\ell$ be positive integers with $\ell \geq k$. For an $\ell$-colorable graph $G$ which is not $k$-choosable, does there exists a partition $\Lambda$ of $k$ such that $G$ is $\Lambda$-choosable?
%is it $\lambda$-choosable for some other partitions $\lambda$ of $k$. 
\end{question}

For a positive integer $k$, we have the following particular refinement chain.
$$\{k\} \succeq \{1, k-1\} \succeq \{1, 1, k-2\} \succeq \cdots \succeq \{1, 1, \cdots, 1\}$$
In the context of Question \ref{ques: l-k_choosable}, we consider the very first refinement $\{k\} \succeq \{1, k-1\}$ in the above chain. If $G$ is $\{1, k-1\}$-choosable then the vertex set of $G$ can be decomposed as $(A, B)$ such that $G[A]$ is $1$-choosable and $G[B]$ is $(k-1)$-choosable. Note that $1$-choosability of $G[A]$ is equivalent to $A$ being an independent set in $G$. Based on these observations, we define {\it near $k$-choosable} graphs.  
A graph $G$ is said to be \emph{near-$k$-choosable} if there exists a decomposition $(A, B)$ of the vertex set such that $A$ is an independent set and the induced subgraph $G[B]$ is $(k-1)$-choosable. The decomposition is called the \emph{near $k$-choosable decomposition} of $G$. Now
we introduce a new decision problem called \texttt{Near-$k$-Choosability}. 
Given a graph $G$, the \texttt{Near-$k$-Choosability} is to decide whether the graph $G$ is near $k$-choosable. This also leads us to define an optimization problem named \texttt{Min-Near-$k$-Choosability} associated with \texttt{Near-$k$-Choosability}. Given a graph $G$, the objective is to find an independent set $A$ in $G$ of minimum cardinality such that $G[V_G \setminus A]$ is $(k-1)$-choosable.
The aforementioned problems specifically assume the set $A$ in the decomposition $(A,B)$ to be independent. It is worthwhile to explore an optimization problem that relaxes this ``independent'' condition on set $A$. We refer to this revised version of \texttt{Min-Near-$k$-Choosability} as \texttt{Min-$k$-Choosable-Del}. Given a graph $G$, in \texttt{Min-$k$-Choosable-Del} it is required to find a vertex set $A \subseteq V_G$ such that $G[V_G \setminus A]$ is $k$-choosable.

\subsection{Known Results}
A graph $G$ is $\triangle$-free if it has no induced $K_3$ as a subgraph.
$\triangle$-free planar graphs are 3-colorable due to the Grötzsch theorem, and such a coloring can be computed in polynomial time \cite{grotzsch1959}. But Voigt \cite{Voigt}, Gutner \cite{Gutner}, Glebova et al. \cite{Gleb} gave examples of $\triangle$-free planar graphs which are not 3-choosable. Moreover, the examples presented contain an independent set whose complement induces a forest. However, the 1-common 3-list colorability (equivalently, near 3-choosability)  of $\triangle$-free planar graphs remains an open problem that is yet to be resolved \cite{tcommon,Zhu}. In \cite{golovach2017survey}, it is proved that $k$-list coloring problems are polynomial time solvable for $k\leq 2$ and $\mathbb{NP}$-complete for $k\geq 3$.

For $k=3$, the problems \texttt{Near-$k$-Choosability} and the decision version of \texttt{Min-Near-$k$-Choosability} are in class $\mathbb{NP}$ due to the characterization of $2$-choosable graphs given by Erd\"os et al. \cite{erdos} involving the $\emph{core}$ of graphs. 
The core of a graph $G$ is obtained by the successive deletion of degree one vertices. A $\theta_{l,m,n}$-graph consists of two distinguished vertices $u$ and $v$ together with three vertex disjoint paths of length $l,m$ and $n$  sharing common endpoints $u$ and $v$. 

\begin{theorem}{\label{thm:2choosablethm}}\cite{erdos}
    A graph is $2$-choosable if and only if, the core of $G$ belongs to the set $\mathcal{C}=\{K_1,C_{2m+2},\theta_{2,2,2m}:m\geq1\}$.
 \end{theorem}

The other known problems which are similar to the above mentioned problems are feedback vertex set, independent feedback vertex set and minimum (independent) feedback vertex set problems. A feedback vertex set $A$ of a graph $G$ is a subset of $V_G$ such that $G[V_G \setminus A]$ is a forest. The problem of finding such a set $A$ of minimum cardinality is known as minimum feedback vertex set problem and is well studied in the literature(see \cite{becker1996optimization,festa1999feedback,speckenmeyer1988feedback}). %This problem is originally known as the \texttt{Feedback Vertex Set Problem}. 
Further, if we impose the condition on the set $A$ to be an independent set, then it is known as minimum independent feedback vertex set problem (also known as Near Bipartite problem) \cite{Misra,agrawal2017improved}.
In \cite{diam4,BONAMY}, the decision problem for \emph{Near Bipartite}
was shown to be $\mathbb{NP}$-complete for graphs with diameter $3$. Furthermore, Bonamy et. al. \cite{BONAMY} proved that for bipartite planar graphs, the minimum independent feedback vertex set problem does not admit a polynomial time algorithm that approximates within a factor of $n^{1-\epsilon}$ for any fixed $\epsilon>0$, unless $\mathbb{P=NP}$ \cite{TCS}. 

\subsection{Our contribution}
If a graph $G$ is near bipartite, then it is also near 3-choosable. It is known that if $G$ is a graph of diameter at most 2 then near bipartiteness of $G$ can be tested in polynomial time. Therefore, \texttt{Near-3-Choosability} is in $\mathbb{P}$ for graphs of diameter at most 2. However, in this paper we prove that \texttt{Near-3-Choosability} is $\mathbb{NP}$-complete for graphs with diameter 3. 
Our main result shows that \texttt{Near-$3$-Choosability} is $\mathbb{NP}$-complete, for $\triangle$-free $4$-colorable graphs of diameter 3, hereby generalizing \cite[Theorem 3.6]{diam4}. The proof is motivated from the construction given in \cite[Theorem 1]{BONAMY} for graphs of diameter at least $3$. As a consequence, we obtain that decision problem for $3$-choosability of $\triangle$-free $4$-colorable graphs having diameter $3$ is $\mathbb{NP}$-complete.

 It is easy to see that bipartite graphs are indeed near $3$-choosable. This leads to the \texttt{Min-Near-$3$-Choosability} i.e. minimizing the size of the independent set in the near $3$-choosable decomposition. We prove that the \texttt{Min-Near-$3$-Choosability} does not admit a polynomial time algorithm in the factor of $n^{1-\epsilon}$ even for bipartite planar graphs.
 
We consider the complexity of \texttt{Min-2-choosable-Del} which is a relaxed version of \texttt{Min-Near-3-Choosability}. Given a graph $G$, in \texttt{Min-2-choosable-Del} it is asked to find a set $A \subseteq V_G$ of minimum cardinality such that $G[V_G\setminus A]$ is 2-choosable. We prove that it is $\mathbb{NP}$-complete and approximable within a factor of $O(\log |V_G|)$.

\section{Complexity of \texttt{Near-$3$-Choosablility}}
In this section, we will prove that \texttt{Near-$3$-Choosability} is $\mathbb{NP}$-complete for $\triangle$-free, 4-colorable graphs with diameter 3. This proof is a polynomial time reduction from \texttt{3-SAT} which is known to be $\mathbb{NP}$-complete \cite{garey1979computers}. Given a 3CNF formula $\varphi = C_1 \vee C_2 \vee \cdots \vee C_k$ over $n$ boolean variables $x_1, x_2, \cdots, x_n$, in  \texttt{3-SAT} it is required to decide whether there exists a truth assignment on these $n$ variables that satisfies $\varphi$ (with at least one true literal in each clause).

\begin{theorem}
 \texttt{Near-$3$-Choosability} is $\mathbb{NP}$-complete for graphs which are $\triangle$-free, 4-colourable and diameter 3.
\end{theorem}
\begin{proof}
% We prove that \texttt{1-in-3-SAT} is polynomially-reducible to \texttt{Near-$3$-Choosabe}. 
Given an instance $\varphi$ of \texttt{3-SAT}, in polynomial time, we construct a graph $H_{\varphi}$ which is $\triangle$-free, 4-colorable and has diameter 3, such that $\varphi$ is satisfiable if and only if $H_{\varphi}$ has a near-3-choosable decomposition $(A,B)$.

Let $\varphi$ be an instance of a \texttt{3-SAT} with $k$ clauses $\iC_1,\iC_2,\dots,\iC_k$ over the variables $x_1,x_2,\dots,x_n$. We may assume that each clause consists of exactly three distinct literals.
%We now construct a graph $H_{\varphi}$ such that the instance $\varphi$ is satisfiable if and only if the graph $H_{\varphi}$ is near-$2$-choosable. 

First we define a constraint graph $P=(V_P, E_P)$ consisting of 17 vertices $V_P=\{v_1, v_2, v_3, w_1, \cdots, w_{14}\}$ as described in Fig. \ref{fig:gadget}.

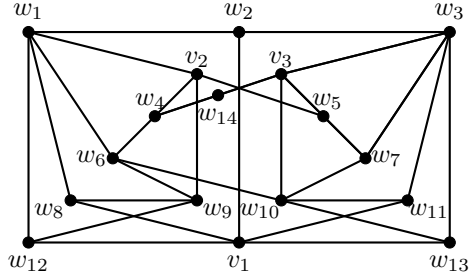
\begin{figure}[h]
    \centering
    \begin{tikzpicture}[scale=1.4]
        %%Vertices
        \draw[fill=black] (0,0) circle (1.5pt); %w_12
        %\draw[fill=black] (0,1) circle (1.5pt);
        \draw[fill=black] (0,2) circle (1.5pt); % w_1
        \draw[fill=black] (2,0) circle (1.5pt); %v_1
        \draw[fill=black] (2,2) circle (1.5pt);
        \draw[fill=black] (0.4,0.4) circle (1.5pt);
        \draw[fill=black] (0.8,0.8) circle (1.5pt);
        \draw[fill=black] (1.2,1.2) circle (1.5pt);
        \draw[fill=black] (1.6,1.6) circle (1.5pt);
        \draw[fill=black] (1.6,0.4) circle (1.5pt);  %w_9  
        \draw[fill=black] (1.8,1.4) circle (1.5pt); %w_14
                 
        \draw[fill=black] (4,0) circle (1.5pt);
        \draw[fill=black] (4,2) circle (1.5pt);
        \draw[fill=black] (2.4,1.6) circle (1.5pt);
        \draw[fill=black] (2.8,1.2) circle (1.5pt);
        \draw[fill=black] (3.2,0.8) circle (1.5pt);
        \draw[fill=black] (3.6,0.4) circle (1.5pt);
        \draw[fill=black] (2.4,0.4) circle (1.5pt);

        %%% Vertex label
        \node at (1.8,1.2) {$w_{14}$};
        \node at (0,-0.2) {$w_{12}$};
        \node at (0,2.2) {$w_1$};
        \node at (2,-0.2)  {$v_1$};
        \node at (2,2.2) {$w_2$};
        \node at (0.2,0.3) {$w_8$};
        \node at (0.6,0.8) {$w_6$};
        \node at (1.16,1.34) {$w_4$};
        \node at (1.6,1.74)  {$v_2$};
        \node at(1.8,0.3) {$w_9$};                  
                 
        \node at (4,-0.2) {$w_{13}$};
        \node at (4,2.2) {$w_3$};
        \node at (2.4,1.74)  {$v_3$};
        \node at (2.84,1.34) {$w_5$};
        \node at (3.4,0.8) {$w_7$};
        \node at (3.8,0.3) {$w_{11}$};
        \node at (2.2,0.3) {$w_{10}$};

        % \node at (0.75,0.70) {\footnotesize $u_1$}; 
        % \node at (0.65,0.12) {\footnotesize $u_{2k-1}$};
        
        % \node at (0.5,0.5) {$w_8$};
        % \node at (3.5,0.5) {$w_9$};
        %%% Edge
        \draw[thick] (0,0) -- (2,0) -- (2,2) -- (0,2) -- (0,0);
        \draw[thick] (2,2) -- (4,2) -- (4,0) --(2,0);
        \draw[thick] (0.4,0.4) -- (0,2) --(0.8,0.8) -- (1.2,1.2)--(1.6,1.6) --(0,2);
        \draw[thick] (1.6,0.4) -- (0.4,0.4);
         \draw[thick] (1.6,0.4) -- (0.8,0.8);
          \draw[thick] (1.6,0.4) -- (1.6,1.6);
           \draw[thick] (1.6,1.6) -- (2.8,1.2);

         \draw[thick] (2.4,1.6) -- (4,2)--(3.2,0.8)--(2.8,1.2)--(2.4,1.6)--(1.2,1.2);
         \draw[thick] (2.4,1.6)--(2.4,0.4) -- (3.6,0.4)--(4,2);
         \draw[thick] (2.4,1.6) -- (3.2,0.8)--(4,2)--(2.4,1.6)--(2.4,1.6)--(1.2,1.2);
         \draw[thick] (0.8,0.8)--(2.4,0.4) -- (3.2,0.8);
         \draw[thick] (2.4,0.4) -- (4,0);
         \draw[thick] (1.6,0.4) -- (0,0);
         \draw[thick,] (0.4,0.4) --  (2,0) -- (3.6,0.4);
         
        %\draw[dashed] (3,1) to[out = -80,in = 150] (4,0);
    \end{tikzpicture}
    \caption{Constraint graph $P$}
    \label{fig:gadget}
\end{figure}

\begin{lemma}{\label{lemma-constraint-graph}}
    Consider the subset $U = \{v_1,v_2,v_3\}$ of $V_P$. Then we have the following:
    \begin{enumerate}[label=(\alph*)]
        \item $P$ is not a bipartite graph.
        \item If $(A,B)$ is a near-$3$-choosable decomposition of $P$, then $U \not\subseteq A$.
        \item Every proper subset $I$ of $U$ can be extended to an independent set $A$ such that $A \cap U = I$ and  $(A,V_P \setminus A)$ is a near-$3$-choosable decomposition of $P$.
    \end{enumerate}
\end{lemma}

\begin{proof}
    \begin{proofcnt}
        \proofcount{$P$ is not bipartite as it contains an odd cycle $\langle v_2, w_5, v_3,$  $w_{14}, w_4\rangle$ of 5 vertices .}

        \noindent 
        \proofcount{Let $(A,B)$ be a near-$3$-choosable decomposition of $P$. Suppose on the contrary that $U \subseteq A$. Note that $\overline{U}:=\{v_1,v_2,v_3,w_6,w_{7}\}$ is the unique maximal independent set containing $U$. 
        %This implies that $A \subseteq \overline{U}$. 
        However, the core of the induced subgraph $P[V_P\setminus\overline{U}]$ does not lie in $\{K_1,C_{2m+2},\theta_{2,2,2m}:m\geq1\}$ contradicting that $(A,B)$ is a near-$3$-choosable decomposition of $P$. }

        \noindent
%        \proofcount{The proper nonempty subsets $\emptyset,\{v_1\},\{v_2\},\{v_3\},\{v_1,v_2\},\{v_1,v_3\}$ and $\{v_2,v_3\}$ can be respectively extended to $\{w_1,w_3,w_4,w_5,w_9,w_{10}\}$, $\{v_1,w_1,w_3,w_4,w_5,w_9,w_{10} \},\{v_2,w_2,w_6,w_7,w_8,$ $w_{11},w_{12},$ $w_{13}\},$ $\{v_3,w_2,w_6,w_7,w_8,w_{11},w_{12},w_{13}\}$$,\{v_1,v_2,w_3,w_{10},w_{14}\}$$, \{v_1,v_3,$ $w_1,w_7,w_9\}$ and $\{v_2,v_3,w_2,w_6,w_7,w_{12},w_{13}\}$. Furthermore, each of these extended subsets is an independent set with its complement inducing a $2$-choosable subgraph of $P$. }
        \proofcount{The proper subsets $\emptyset,\{v_1\},\{v_2\},\{v_3\},\{v_1,v_2\},\{v_1,v_3\}$ and $\{v_2,v_3\}$ can be respectively extended to independent sets
        \begin{align*} 
       		&\{w_1,w_3,w_4,w_5,w_9,w_{10}\}, \{v_1,w_1,w_3,w_4,w_5,w_9,w_{10} \}\\ &\{v_2,w_2,w_6,w_7,w_8, w_{11},w_{12}, w_{13}\}, \{v_3,w_2,w_6,w_7,w_8,w_{11},w_{12},w_{13}\},\\
       		&\{v_1,v_2,w_3,w_{10},w_{14}\}, \{v_1,v_3,w_1,w_7,w_9\} \text{ and } \{v_2,v_3,w_2,w_6,w_7,w_{12},w_{13}\}.
 	    \end{align*}
        Furthermore, each of these extended subsets satisfies the required condition. }
    \end{proofcnt}  
\end{proof}

For each clause $C_s$, we define a clause gadget $\mathcal{J}_s$ which is an array of $n+14k +1$ rows, and $17$ columns. In each row except the last one, every position consists of exactly two vertices, which are referred to as the true vertex and the false vertex. We say that these pairs of vertices are mates. Pictorially (see Fig. \ref{fig:npcomplete}), the true and false vertices are drawn in black and white, respectively. Each such row is isomorphic to $K_{17,17}$ minus a perfect matching consisting of $17$ mate edges. 
The first $n$ rows form the variable block wherein the $i$th row corresponds to the variable $x_i$ in $\varphi$. The next $14k$ rows are divided into $k$ blocks, each one consisting of $14$ rows. We will refer these blocks as ${J}_{s,1}, {J}_{s,2}, \cdots, {J}_{s,k}$. In the last row, each column consists of a single vertex, which is adjacent to every other vertex in its column. We will denote this row as dominating row. Now it remains to introduce further edges of the constraint graph $P$. Let the literals of $C_s$ be $l_{s1}, l_{s2}, l_{s3}$.  We choose the first three columns of the variable block of $\mathcal{ J}_s$ to represent these literals. If $l_{sr}$ is the literal associated with the variable $x_i$, then we choose $V_r^s$ as the vertex from $i$th row and $r$th column, and choose the true vertex if the literal is positive, otherwise the false vertex. For $j \in \{4, 5, \cdots, 17\}$, let $W_j^s$ be the true vertex in the $(j-3)$th row and $j$th column in the clause block ${J}_{s,s}$. We have identified 17 vertices from the clause gadget $\mathcal{J}_s$ as $\{V_1^s, V_2^s, V_3^s, W_1^s, \cdots, W_{14}^s\}$ and introduce the 31 edges of constraint graph $P$ such that the induced graph (of these 17 vertices) is isomorphic to $P$.

Now, we construct the  graph $H_{\varphi}$ by merging the $\mathcal{J}_s$ graphs as follows.\\ 
$(i)$ Make a disjoint union of clause graphs $\mathcal{J}_s$, $1 \leq s \leq k$. \\
$(ii)$ Introduce the edges from each true vertex of each clause graph to each false vertex in the same row of other clause graphs. \\
$(iii)$ Introduce a new additional vertex $d_0$ and make it adjacent to each vertex in the last row of each clause gadget. 

It can be observed that each row in $H_{\varphi}$ has $17k$ columns.
We denote $Y_i$ and $Z_i$ as the set of black vertices and white vertices, respectively, in $i^{th}$ row, that is, $Y_i =\{y_{i1}, y_{i2}, \cdots, y_{i(17k)}\}$ and $Z_i =\{z_{i1}, z_{i2}, \cdots, z_{i(17k)}\}$. Further, we denote the vertices in the dominating row as $d_1,d_2,...,d_{17k}$.

\begin{figure}
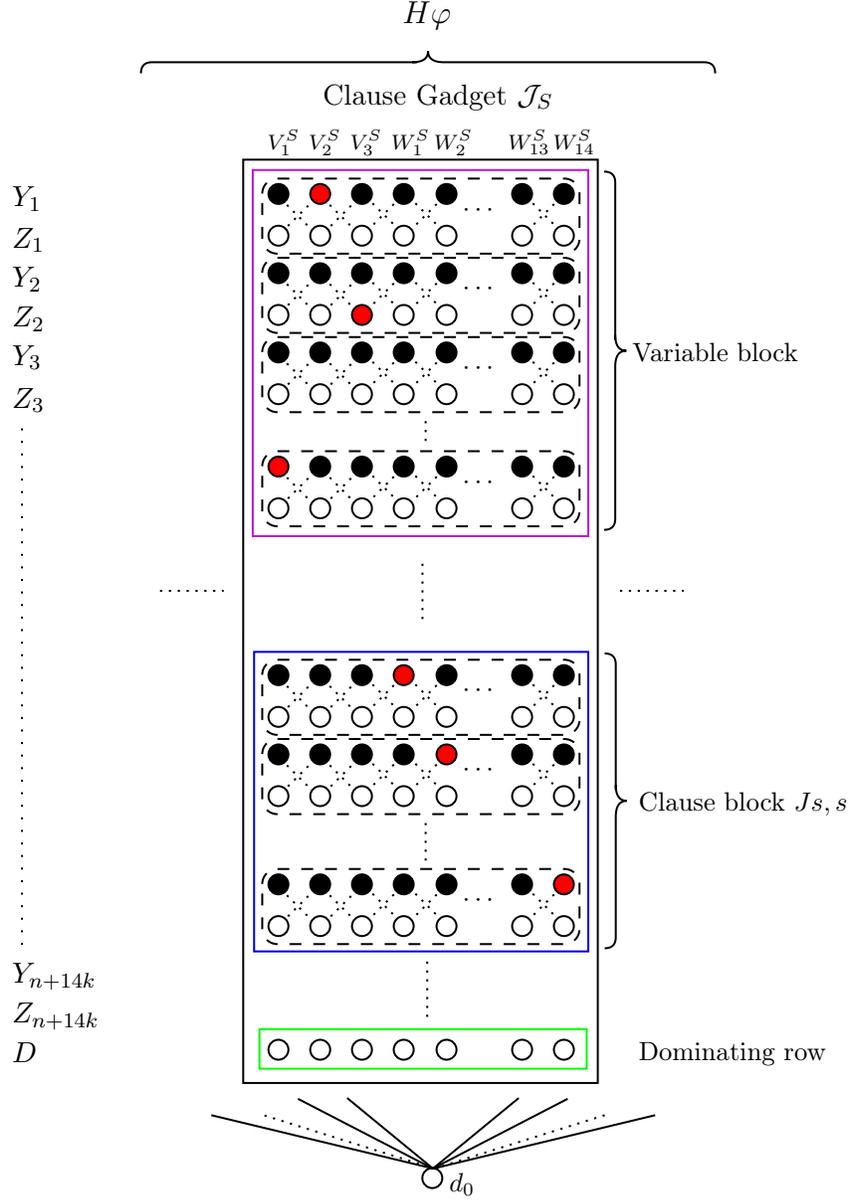

    \centering
    \tikzset{every picture/.style={line width=0.75pt}} %set default line width to 0.75pt        



    \caption{ An illustration of the graph $H_{\phi}$ is presented, with a particular focus on the $s^{th}$ clause gadget.} \label{fig:npcomplete}
\end{figure}

\begin{lemma}
The graph $H_{\varphi}$ is $\triangle$-free, 4-colorable and of diameter 3.
\end{lemma}
\begin{proof}
It can be observed that $H_{\varphi}$ is $\triangle$-free. $H_{\varphi}$ is of diameter 3, because each clause gadget $\mathcal{ J}_s$ along with vertex $d_0$ (even without the constraint graph $P$) is of diameter 3.
We define a proper 4-coloring of the graph $H_{\varphi}$ as follows. For each $i\in \{1,2,\dots,n\}$, assign color $1$ to all vertices in $Y_i$ and color $2$ to all vertices in $Z_i$. This is a proper $2$-coloring on the induced graph on the variable block of $H_\varphi$. Now we extend this coloring to clause blocks of $H_\varphi$. It is easy to observe that the set $S=\{w_1,w_3,w_4,w_9,w_{10}\}$ is the maximal independent set of the constraint graph $P$ for which each vertex has exactly one neighbor in the set $U=\{v_1,v_2,v_3\}$. For $i\in \{n+1,2,\dots,n+14k\}$, there is exactly one vertex $y_{ij}$ in $Y_i$, which is identified with one of the vertices $w_1,w_2,\dots,w_{14}$ in the constraint graph $P$. If the vertex $y_{ij}$ is identified with one of the vertices in $S$, then $y_{ij}$ is adjacent to exactly one vertex, say $v$, in the variable block. In this case, color all vertices of $Y_i$ color $1$ and all vertices of $Z_i$ color $2$ if the vertex $v$ is colored $2$. Otherwise, color all the vertices of $Y_i$ color $2$ and the vertices $Z_i$ color $1$. On the other hand, if the vertex $y_{ij}$ is identified with a vertex of $P$ which is not in $S$, then color the vertices of $Y_i$ color $3$ and color the vertices of $Z_i$ color $2$. Now color all the vertices of the dominating row color $4$ and $d_0$ by any one of the three colors $1,2$ or $3$. By the construction of $H_\varphi$, this is a proper $4$-coloring of $H_\varphi$.  
\end{proof}

Suppose that $\varphi$ has a satisfying truth assignment $\tau$. We now proceed to construct a decomposition $(A, B)$ from the truth assignment $\tau$. Let the vertex $d_0$ lie in $A$, and its neighbors in $B$. If $\tau(x_i)$ is false, then let $B$ contain all the vertices $Z_i$. Otherwise, let the vertices in $Y_i$ lie in $B$. In either case, let $A$ contain the mates of the vertices lying in $B$. Consider the induced graph $P$ in each clause gadget $\mathcal{ J}_s$. The vertices $V^S_1, V^S_2, V^S_3$ associated with the variables of $C_s$ either lie in $A$ or $B$, wherein at most two of them that correspond to a false literal lie in $A$ (since $C_s$ is satisfiable under the truth assignment $\tau$ ). Let $U$ be such a set of vertices associated with literals that are false in $C_s$. By Lemma \ref{lemma-constraint-graph}, we extend the vertex set $U$ to an independent set (by choosing the corresponding vertices from the clause block) that gives near $3$-choosable decomposition of $P$. Whenever a vertex of a clause block $J_{s,s}$ is assigned to $A$ or $B$, all the other true vertices in the same row of $H_{\varphi}$ are assigned to the same set, and their mates are assigned to the other set. This process assigns every vertex of  $H_{\varphi}$ to exactly one of $A$ or $B$.
This construction implies that $A$ is an independent set. 

Now we show that the core of the subgraph induced by $B$ lies in $\mathcal{ C}$. Note that the vertices in the bottom-most row, which lie in $B$ have an edge incident to every vertex in the same column. Furthermore, from each row, either every true vertex or every false vertex lie in $B$. This implies that the cycles (if exist) in $H_{\varphi}[B]$, must lie in the same clause block. Let $\mathcal{ J}_s[B]$ be the subgraph induced by $B$ in the clause gadget $\mathcal{ J}_s$. The true and false vertices not lying in the constraint graph have degree 1 in $\mathcal{ J}_s[B]$ due to the edge incident to the dominating row. Also, each vertex in the dominating row has at most one vertex of degree more than 1 in $\mathcal{ J}_s[B]$(since it has only one neighbor in the constraint graph). This implies that the core of $\mathcal{ J}_s[B]$ is equal to the core of the induced subgraph of the constraint graph in each clause gadget. This implies that connected components of the core of $H_{\varphi}[B]$ must lie in $\mathcal{ C}$.

Conversely, assume that $H_{\varphi}$ is near $3$-choosable. Let $(A,B)$ be the decomposition of $H_{\varphi}$ where $A$ is an independent set and $H_{\varphi}[B]$ is $2$-choosable. We will construct a satisfying truth assignment $\tau$ of $\varphi$. Our first step is to observe some properties of the graph $H_{\varphi}$ under the assumption that $(A, B)$ is a near $3$-choosable decomposition of $H_{\varphi}$.
 
At least one vertex in $Y_i$ or $Z_i$ belongs to $A$, for each $i \in \{1,2,\dots,n+14k\}$. Otherwise, $Y_i \cup Z_i \subset B$, but the induced graph on these vertices is $K_{17k,17k}$ minus a perfect matching which is not $2$-choosable. Now, we prove that if a vertex in $Y_i$ is in $A$, then $Y_i$ can have at most one vertex from $B$ (a similar statement holds for $Z_i$).  Suppose that there exists $y_{ij},y_{il},y_{im} \in Y_i$ such that $y_{ij} \in A$ and $y_{il},y_{im}\in B$. Then all the vertices in $Z_i$ except $z_{ij}$ belong to $B$ since the vertex $y_{ij} \in A$. This implies that $\{y_{il},y_{im} \} \cup Z_i \setminus \{z_{ij}\}$ is a subset of $B$ and hence the induced graph $H_{\varphi}[\{y_{il},y_{im} \} \cup Z_i \setminus \{z_{ij}\}]$ is $2$-choosable. This is a contradiction as it contains $K_{2,4}$ as a subgraph which is not 2-choosable. Therefore, for any  $i \in \{1,2,\dots,n+14k\}$, if $Y_i$ or $Z_i$ contains a vertex from $A$, they cannot contain more than one vertex from $B$.

Suppose that there exists an $i \in \{1,2, \dots, n+14k\}$ such that one of $Y_i$ or $Z_i$ contains a vertex in $A$ and a vertex in $B$ in the decomposition of $H_\varphi$. Without loss of generality, assume that there exists $y_{ij}$ and $y_{il}$ in $Y_i$ such that $y_{ij} \in A$ and $y_{il} \in B$. Since $Y_i$ can not contain more than one vertex from $B$, we have $Y_i\setminus\{y_{il}\} \subset A$. This implies that all vertices in $Z_i$ lie in $B$.
If the vertex $d_0\in B$, then $\{y_{il},d_0\}\cup Z_i \cup \ D\setminus \{d_l\} \subseteq B$. This contradicts that $H_{\varphi}[B]$ is $2$-choosable  because it contains $\theta_{3,3,3}$ as a subgraph induced by vertices $y_{il}, d_0,$ some three vertices from $Z_i$ (other than $z_{il}$) and their corresponding three vertices from the dominating row. 
Therefore $A$ must contain the vertex $d_0$. Consider the following proper $3$-coloring on $H_{\varphi}$ where the vertices in $A$ are colored red, and the vertices in $B$ are colored blue and green (since $H_{\varphi}[B]$ is $2$-choosable). Without loss of generality, assume that the vertex $y_{il}$ is colored blue. Since the induced graph on $Y_i \cup Z_i$ is $K_{17,17}$ minus a perfect matching, all vertices except $z_{il}$ get the color green. This shows that all the vertices in the dominating row except $d_l$ are colored blue. Thus all the vertices in $H_{\varphi}$ other than the vertices in the $l^{\mathrm{th}}$ column, dominating row and $d_0$, are colored green or red. Then at least one copy of the constraint graph $P$ admits a proper coloring with colors red and green only. This is a contradiction because the constraint graph is not bipartite (by Lemma \ref{lemma-constraint-graph}). Therefore, for any $i \in \{1,2, \cdots, n+14k\}$, either $Y_i$ or $Z_i$ is contained in $A$.

Now we define a truth  assignment $\tau$ of $\varphi$ as follows. For every $i \in \{1,2, \cdots, n\}$, we set $x_i=0$ in $\tau$ if all vertices in $Y_i$ is in $A$; otherwise if all vertices of $Z_i$ are in $A$ then we set $x_i=1$ in $\tau$. We prove that $\tau$ is a satisfying assignment of $\varphi$. Consider a clause $C_s=(l_{s1} \vee l_{s2} \vee l_{s3})$ of $\varphi$, where $l_{sr} \in \{ x_{i},\overline{x_i}\}$ for some $i \in \{1,2, \cdots, n\}$ and $r \in \{1,2,3\}$. This clause $C_s$ corresponds to the clause gadget $\mathcal{J}_s$ in $H_\varphi$. The vertices of this clause gadget is also partitioned into $A'$ and $B'$ where $A'\subset A$ and $B'\subset B$, since $(A,B)$ is near $3$-choosable decomposition of $H_{\varphi}$.  By Lemma \ref{lemma-constraint-graph}, not all the three vertices $V_1^s,V_2^s,V_3^s$ lie in the set $A$. Also by the construction of $H_\varphi$ and definition of $\tau$, the literal $l_{sr}=0$ under $\tau$ if and only if the vertex $V_r^s$ is in $A$, for $r\in\{1,2,3\}$. Since not all the vertices $V_1^s,V_2^s,V_3^s$ lie in $A$, implies that the literals $ l_{s1},l_{s2},l_{s3}$ of $C_s$ are not all false in $\tau$. So the clause $C_{s}$ is satisfied by $\tau$. Therefore $\tau$ is a satisfying truth assignment of $\varphi$.  
\end{proof}

%%%%%%%%%%%%%%%%%%%%%%%%%%%%%

It is an open problem to decide whether a planar triangle-free graph is near 3-choosable or not \cite{Zhu,tcommon}. However, we could able to prove that it is $\mathbb{NP}$-complete to decide near 3-choosability of a 4-colorable, triangle-free graph of diameter 3.

\section{Inapproximability of \texttt{Min-Near-3-Choosability}}
This section focuses on the approximability of \texttt{Min-Near-3-Choosability}, a problem that aims to find an independent set $A$ in a graph $G$ such that $G[V_G \setminus A]$ is $2$-choosable. However, not all graphs have a feasible solution for this problem, as some graphs do not have an independent set $A$ that satisfies the conditions. Nevertheless, \texttt{Min-Near-3-Choosability} is well-defined for bipartite graphs, which are always near $3$-choosable. Given a bipartite graph $G$, we can define a near $3$-choosable decomposition $(A, B)$ of $V_G$ as one where $A$ is an independent set and $G[B]$ is $2$-choosable. Here, $A$ is referred to as the \emph{near $3$-choosable deletion} set in $G$. Furthermore, if none of the proper subsets of $A$ satisfies the near $3$-choosable conditions, then $A$ is known as a \emph{minimal near $3$-choosable deletion set}.

A strong inapproximability result for \texttt{Min-Near-3-Choosability} is presented in Theorem \ref{thm-inappx}, even when restricted to planar bipartite graphs. Our proof of this theorem involves a reduction from the well-known \texttt{Planar-3-SAT}, which is the 3-satisfiability problem with the additional condition that the associated graph of a given instance $\varphi$ is planar. We denote the set of variables and clauses in $\varphi$ as $X=\{x_1, \cdots, x_n\}$ and $C=\{C_1, C_2, \cdots, C_k\}$, respectively. The bipartite graph $P_{\varphi}=(X\cup C, E)$ has vertex set $X \cup C$ with an edge $(x_i, C_j) \in E$ if and only if the variable $x_i$ appears in $C_j$. We classify such an edge $(x_i, C_j)$ as a {\it positive edge} if $x_i$ appears positively in $C_j$; otherwise, it is called a {\it negative edge}. We say that $\varphi$ is an instance of \texttt{Planar-3-SAT} if the associated graph $P_{\varphi}$ is planar, which is known to be $\mathbb{NP}$-complete \cite{lichtenstein1982planar}.
    
\begin{theorem} \label{thm-inappx}
For planar bipartite graphs,  \texttt{Min-Near-3-Choosability} can not be approximated  within a factor of $n^{1-\epsilon}$, for $\epsilon>0$, unless $\mathbb{P=NP}$.
\end{theorem}

\begin{proof}
Given an instance $\varphi$ of \texttt{Planar-3-SAT}, in polynomial time, we construct a planar bipartite graph $G_{\varphi,p}$ (integer $p$ will be defined later) corresponding to $P_{\varphi}$ by replacing its vertices and edges by suitable gadgets. 

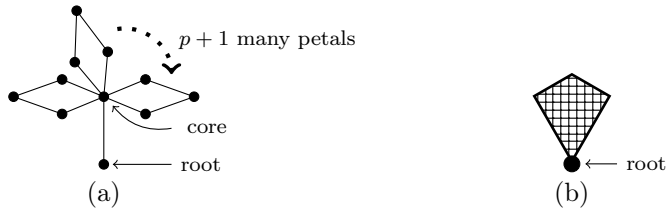
\begin{figure}[h!]
    \centering
    \begin{tikzpicture}[scale=0.6]
        \draw[fill=black] (0:0) circle (3pt);
        \draw[fill=black] (0:2) circle (3pt);
        \draw[fill=black] (-22.5:1) circle (3pt);
        \draw[fill=black] (22.5:1) circle (3pt);

        \draw[fill=black] (180:2) circle (3pt);
        \draw[fill=black] (157.5:1) circle (3pt);
        \draw[fill=black] (202.5:1) circle (3pt);

        \draw[fill=black] (130:1) circle (3pt);
        \draw[fill=black] (85:1) circle (3pt);
        \draw[fill=black] (107.5:2) circle (3pt);

        \draw[fill=black] (-90:1.5) circle (3pt);
        
        \begin{pgfonlayer}{background layer}
            \draw (-90:1.5) -- (0:0) -- (-22.5:1) -- (0:2) -- (22.5:1) -- (0,0) -- (85:1) -- (107.5:2) -- (130:1) -- (0:0) -- (157.5:1) -- (180:2) -- ( 202.5:1) -- (0:0) ;

            \draw[->,ultra thick,loosely dotted] (0.3,1.5) to[bend left=50] node[midway,right] {\footnotesize \; $p+1$ many petals} (1.5,0.5);
            \draw[<-] (0.2,-0.2) to[bend right = 30] (1.5,-0.7) node[right] {\footnotesize \;core};
            \draw[<-] (0.2,-1.5) -- (1.5,-1.5) node[right] {\footnotesize root};
            \node [below=1cm] at (0,0)
        {
            (a)
        };
            
        \end{pgfonlayer}
    \end{tikzpicture}    
    \hspace{2cm}
    \begin{tikzpicture}[scale=0.6, petal/.style={kite,draw,pattern=grid, line width=1pt}]
        \node[petal,minimum size =1cm, anchor=south] at (0,0){} ;
        \draw[fill=black] (0,0) circle(5pt);
        \draw[<-] (0.3,0) -- (1,0) node[right] {\footnotesize root};
        \node [below=3pt] at (0,0)
        {
            (b)
        };
    \end{tikzpicture}        
    \caption{Forbidden gadget}
    \label{fig:forbidding gadget}

\end{figure}

To begin, we introduce the {\it forbidden gadget} (see Fig. \ref{fig:forbidding gadget}a) which comprises a root vertex, a core vertex, and $p+1$ petals (where each petal is a $C_4$). The root vertex connects to the core vertex, and the core vertex is the sole common vertex of all $p+1$ petals. A simplified depiction of this gadget is shown in Fig. \ref{fig:forbidding gadget}b for clarity. Note that the root vertex and core vertex cannot belong to an independent set simultaneously. The forbidden gadget is not $2$-choosable, and thus any minimal near 3-choosable deletion set must include at least one vertex from the gadget. If the core vertex is present in a minimal near 3-choosable deletion set, the other nodes in this gadget will not be included. On the other hand, if the root vertex is included, then $p$ vertices (one each from distinct petals) must be included in the near 3-choosable deletion set.

% First, we define the {\it forbidden gadget} (refer Figure \ref{fig:forbidding gadget}a). It consists of a root vertex, core vertex, and $p+1$ petals (where each petal is a $C_4$). The root vertex is adjacent to the core vertex, and the core is the unique common vertex of all $p+1$ petals.
% For an easy pictorial representation of this gadget, we will depict it by the structure shown in Figure \ref{fig:forbidding gadget}b. It is easy to observe that both the root vertex and core vertex can not be in an independent set. Since the forbidden gadget is not $2$-choosable, any minimal near 3-choosable deletion set must contain at least one vertex from the forbidden gadget. If the core vertex is in a minimal near 3-choosable deletion set, then other nodes in this gadget will not be included. However, if the root vertex is included, then we must include $p$ vertices (one vertex from each petal) in the near 3-choosable deletion set. 

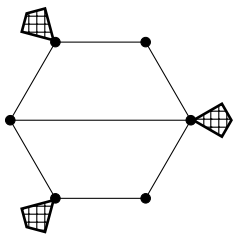
\begin{figure}[h!]
    \centering
        \begin{tikzpicture}[scale=0.6,petal/.style={kite,draw,pattern=grid, line width=1pt}]
           \newdimen\R
            \R=2cm
            \foreach \x in {0,60,120,...,300}
                \draw[fill=black] (\x:\R) circle (3pt);
            \begin{pgfonlayer}{background layer}
                \draw (0:2) \foreach \x in {60,120,...,360} {-- (\x:\R) }; 
                \draw (0:2) -- (180:2);
            \end{pgfonlayer}
             \put(-20,0){\makebox(-20,0)[r]{$c_{j1}$}};
             \put(15,23){\makebox(15,23)[r]{$c_{j2}$}};
             \put(15,-25){\makebox(15,-25)[r]{$c_{j3}$}}; 
            
            \node[petal,minimum size =0.2cm, anchor=south,rotate=270] at (0:2){} ;
	        \node[petal,minimum size =0.2cm, anchor=south,rotate=45] at (120:2){} ;
            \node[petal,minimum size =0.2cm, anchor=south,rotate=135] at (240:2){} ;
        \end{tikzpicture}
    \caption{Clause Gadget}
    \label{fig:clause gadget}
\end{figure}
The clause gadget is defined in Fig. \ref{fig:clause gadget}. If $C_j$ is a clause in $\varphi$ then its main six nodes are denoted by $c_{j1}, c_{j2}, c_{j3}, w_{j1}, w_{j2}, w_{j3}$. It can be observed that this gadget is not 2-choosable and at least one of these six vertices must be included in a minimal near 3-choosable deletion set.

\begin{figure}[h!]
	\centering
\begin{tikzpicture}[scale=0.6, petal/.style={kite,draw,pattern=grid, line width=1pt}]
	\draw[fill=Blue, opacity=1] (0,0) circle (4pt) node[anchor=north]{$c_{jr}$}; 
	\draw[fill=black] (-2,-0.5) circle (3pt);
	\draw[fill=black] (-2,1.5) circle (3pt);
    \draw[fill=black] (2,-0.5) circle (3pt);
	\draw[fill=black] (2,1.5) circle (3pt);
	\draw[fill=red] (0,2) circle (4pt);
    \draw[fill=black] (2,2) circle (3pt);
	\draw[fill=black] (2,4) circle (3pt);
	\draw[fill=black] (2,6) circle (3pt);
	\draw[fill=Blue] (0,4) circle (4pt);
	\draw[fill=red] (0,6) circle (4pt);
	\draw[fill=black] (-2,6) circle (3pt);
	\draw[fill=Blue] (0,8) circle (4pt) node[anchor=south]{$x_{i}$};
	\draw[fill=black] (-2,4) circle (3pt);
	\draw[fill=black] (2,8) circle (3pt);
	% \draw[fill=Blue] (0,10) circle (4pt) node[anchor=south]{$v_{i}$};
	% \draw[fill=black] (2,10) circle (3pt);
	
	\begin{pgfonlayer}{background layer}
    \draw (0,0) -- (-2,-0.5) -- (-2,1.5) -- (0,2) -- (2,1.5) -- (2,-0.5) -- (0,0);
    \draw (0,2) -- (2,2) -- (2,4) -- (0,4) -- (0,6) -- (-2,6);
	\draw (0,6) -- (2,6) -- (2,8) -- (0,8) -- (0,6) -- (0,4) -- (0,2) -- (0,0);
	\draw (-2,6) -- (-2,4) -- (0,4); \draw (2,4) -- (2,6);
    \draw (-2,6) to[bend right =90] (0,2) ;
    \end{pgfonlayer}
	
	\node[petal,minimum size =0.2cm, anchor=south,rotate=115] (T)at (-2,-0.5){} ;
	\node[petal,minimum size =0.2cm, anchor=south,rotate=115] at (-2,1.5){} ;
	\node[petal,minimum size =0.2cm, anchor=south,rotate=-135] at (-2,6){} ;
	\node[petal,minimum size =0.2cm, anchor=south,rotate=-115] at (2,1.5){} ;
    \node[petal,minimum size =0.2cm, anchor=south,rotate=270] at (2,2){} ;
	\node[petal,minimum size =0.2cm, anchor=south,rotate=270] at (2,4){} ;
	\node[petal,minimum size =0.2cm, anchor=south,rotate=-115] at (2,-0.5){} ;
	\node[petal,minimum size =0.2cm, anchor=south,rotate=270] at (2,6) {} ;
	\node[petal,minimum size =0.2cm, anchor=south,rotate=270] at (2,8){} ;
	\node[petal,minimum size =0.2cm, anchor=south,rotate=-45] at (-2,4){} ;
    \node [below=1cm] at (0,0)
        {
            (a) Positive edge gadget $[x_i,c_{jr}]$
        };
\end{tikzpicture}
\hspace{1cm}
\begin{tikzpicture}[scale=0.6,petal/.style={kite,draw,pattern=grid, line width=1pt}]
    \begin{pgfonlayer}{foreground layer}
    
	\draw[fill=Blue] (0,0) circle (4pt) node[anchor=north]{$c_{jr}$};
	\draw[fill=black] (-2,0) circle (3pt);
	\draw[fill=black] (-2,2) circle (3pt);
	\draw[fill=red, fill opacity=1] (0,2) circle (4pt) node[anchor=south]{$x_{i}$};
	\draw[fill=black] (2,2) circle (3pt);
	\draw[fill=black] (2,0) circle (3pt);

	\node[petal,minimum size =0.2cm, anchor=south,rotate=90] at (-2,0){} ;
	\node[petal,minimum size =0.2cm, anchor=south,rotate=90] at (-2,2){} ;
	\node[petal,minimum size =0.2cm, anchor=south,rotate=270] at (2,2){} ;
	\node[petal,minimum size =0.2cm, anchor=south,rotate=270] at (2,0){} ;

    \end{pgfonlayer}

    \node[below=5mm] at (0,0) {(b) Negative edge gadget $[x_i,c_{jr}]$};
    \begin{pgfonlayer}{background layer}
    \draw (0,0) -- (-2,0) -- (-2,2) -- (0,2) -- (2,2) -- (2,0) -- (0,0) -- (0,2) ;
	\end{pgfonlayer}
 
\end{tikzpicture}

    \caption{Positive and Negative Edge Gadgets along with their respective proper min-3-choosable decomposition sets}
    \label{fig:edge gadgets}
\end{figure}
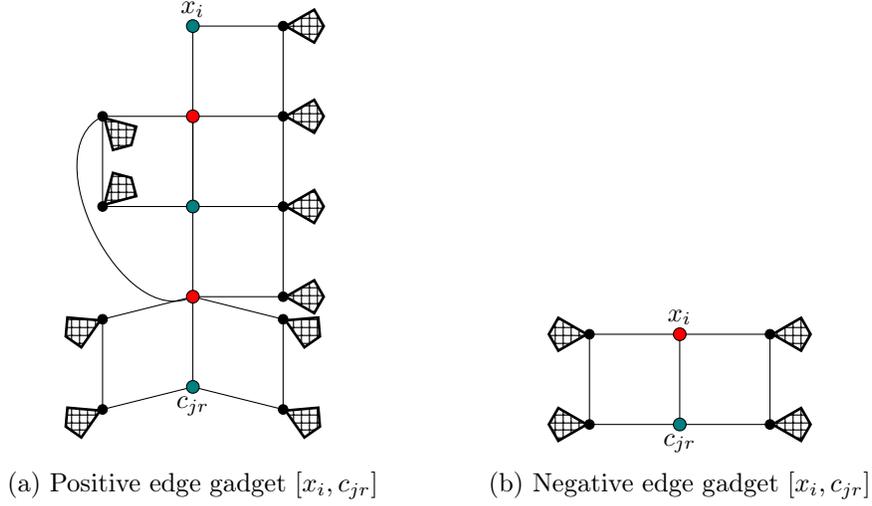
% In Figure \ref{fig:edge gadgets}a and Figure \ref{fig:edge gadgets}b, the gadgets for positive edge and negative edges in $P_{\varphi}$ are illustrated. It is important to note that these gadgets are not 2-choosable. There are exactly two types of minimum near 3-choosable deletion sets for both positive and negative edge gadgets. These two types of sets consist of red or blue vertices along with the core vertices of each constraint gadget. These two types of deletion sets are of sizes strictly smaller than $p$ (this fact will be used later in the proof). Let $(x_i, c_{jr})$ be an edge in $P_{\varphi}$ and $A$ be a minimal independent near 3-choosable decomposition set in the corresponding edge gadget. It can be observed that $x_i \in A$ if and only if $c_{jr} \in A$ for a positive edge; and $x_i \in A$ if and only if $c_{jr} \notin A$ for a negative edge. This is the way the edge gadgets propagate the choice of $x_i$ to $c_{jr}$. Note that the negative edge gadgets propagate the choice oppositely.

The figures in Fig. \ref{fig:edge gadgets}a and Fig. \ref{fig:edge gadgets}b depict the gadgets used for positive and negative edges in $P_{\varphi}$, respectively. It should be emphasized that these gadgets are not 2-choosable. For both types of gadgets, there are two types of minimal near 3-choosable deletion sets, consisting of either red or blue vertices along with the core vertices of each constraint gadget. These deletion sets are of sizes strictly less than $p$ (this fact will be used later in the proof). Suppose $(x_i, c_{jr})$ is an edge in $P_{\varphi}$ and let $A$ be a minimal independent near 3-choosable decomposition set in the corresponding edge gadget. We observe that $x_i$ is included in $A$ if and only if $c_{jr}$ is included in $A$ for a positive edge, and $x_i$ is included in $A$ if and only if $c_{jr}$ is not included in $A$ for a negative edge. This is how the edge gadgets propagate the choice of $x_i$ to $c_{jr}$. Note that the negative edge gadgets propagate the choice oppositely.

%Any minimal near $3$-choosable decomposition set of the respective edge gadgets contains either blue vertices or red vertices.
% \textcolor{red}{It can be observed that for every independent partition set $A$, $v_i \in A$ if and only $c_{jr} \in A$ for the positive edge gadget ($c_{jr}=v_i$ in the clause $C_j$); and $v_i \in A$ if and only if $c_{jr} \notin A$ for the negative edge ($c_{jr}=\overline{v_i}$ in the clause $C_j$).}

We have assumed that $\varphi$ is a 3CNF formula (an instance of \texttt{Planar-3-SAT}) with $n$ variables and $k$ clauses. For the construction of the graph $G_{\varphi, p}$, we define $q_{\varphi} = 280k$ and $p = q_{\varphi}^{\lceil \frac{2-\epsilon}{\epsilon}\rceil}$, where $0< \epsilon \leq 1$ is a fixed constant. It is easy to observe that $p$ is a polynomial in the input parameters $n$ and $k$ of $\varphi$.

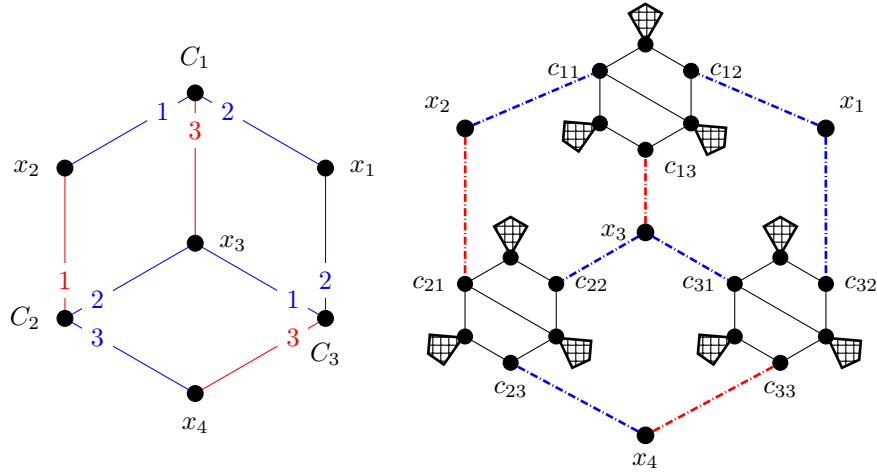
\begin{figure}
    \centering
    \begin{tabular}{m{5cm}m{6cm}}
    \begin{tikzpicture}[scale=1]
        \draw[fill=black] (0:0) circle (3pt)node[anchor=west,right=2mm] {$x_{3}$};;
        \draw[fill=black] (90:2) circle (3pt)node[anchor=south,above=2mm]{$C_{1}$};
        \draw[fill=black] (-90:2) circle (3pt)node[anchor=north,below=2mm]{$x_4$};;
        \draw[fill=black] (150:2) circle (3pt)node[anchor=east,left=2mm]{$x_{2}$};;
        \draw[fill=black] (210:2) circle (3pt)node[anchor=east,left=2mm]{$C_{2}$};;
        \draw[fill=black] (30:2) circle (3pt)node[anchor=west,right=2mm]{$x_{1}$};;
        \draw[fill=black] (-30:2) circle (3pt)node[anchor=north,below=2mm]{$C_{3}$};;

        \begin{pgfonlayer}{background layer}
        \path[draw,red] (90:2) -- (0:0) node [near start, fill=white] {3};
        \path[draw,blue] (90:2) -- (30:2) node [near start, fill=white] {2};
        \path[draw,blue] (90:2) -- (150:2) node [near start, fill=white] {1};
        \path[draw,red] (210:2) -- (150:2) node [near start, fill=white] {1};
        \path[draw,blue] (210:2) -- (0:0) node [near start, fill=white] {2};
        \path[draw,blue] (210:2) -- (270:2) node [near start, fill=white] {3};
        \path[draw,blue] (330:2) -- (0:0) node [near start, fill=white] {1};
        \path[draw,blue] (330:2) -- (30:2) node [near start, fill=white] {2};
        \path[draw,red] (330:2) -- (270:2) node [near start, fill=white] {3};
        
        \end{pgfonlayer}
    \end{tikzpicture}  
    &
    \begin{tikzpicture}[scale=.9, vertex/.style={draw,circle,fill=black,scale=0.3}]
        \pic[scale=0.35,rotate=-30](A) at (90:2) {clause1};
        \pic[scale=0.35,rotate=-30](B) at (210:2.3) {clause2};
        %\node[label=270:{$C_2$}]() at (210:3) {};
        \pic[scale=0.35,rotate=-30](C) at (330:2.3) {clause3};
        %\node[label=270:{$C_3$}]() at (330:3) {};
        \node[vertex,label=180:{$x_3$}](x3) at (0:0) {$x_{3}$};
        \node[vertex,label=270:{$x_4$}](x4) at (-90:3) {$x_{4}$};
        \node[vertex,label=120:{$x_2$}](x2) at (150:3.075) {$x_{2}$};
        \node[vertex,label=60:{$x_1$}](x1) at (30:3.075) {$x_{1}$};        
        \draw[red,densely dash dot, line width=1pt] (x3) -- (A300){};
        \draw[blue,densely dash dot, line width=1pt] (x3) -- (B60){}; % node [midway, fill=white] {\footnotesize 2};    
        \draw[blue,densely dash dot, line width=1pt] (x3) -- (C180){}; % node [midway, fill=white] {\footnotesize 1};
        \draw[blue,densely dash dot, line width=1pt] (x2) -- (A180){}; %node [midway, fill=white] {\footnotesize 1};
        \draw[red,densely dash dot, line width=1pt] (x2) -- (B180){}; %node [midway, fill=white] {\footnotesize 1};
        \draw[blue,densely dash dot, line width=1pt] (x4) -- (B300){}; % node [midway, fill=white] {\footnotesize 3};
        \draw[red,densely dash dot, line width=1pt] (x4) -- (C300){}; %node [midway, fill=white] {\footnotesize 3};
        \draw[blue,densely dash dot, line width=1pt] (x1) -- (A60){}; % node [midway, fill=white] {\footnotesize 2};
        \node[label=0:{$c_{12}$}]() at (A60) {};
        \node[label=180:{$c_{11}$}]() at (A180) {};
        \node[label=345:{$c_{13}$}]() at (A300) {};
        \node[label=0:{$c_{22}$}]() at (B60) {};
        \node[label=180:{$c_{21}$}]() at (B180) {};
        \node[label=270:{$c_{23}$}]() at (B300) {};
        \node[label=0:{$c_{32}$}]() at (C60) {};
        \node[label=180:{$c_{31}$}]() at (C180) {};
        \node[label=270:{$c_{33}$}]() at (C300) {};
        \draw[blue,densely dash dot, line width=1pt] (x1) -- (C60){};% node [midway, fill=white] {\footnotesize 2};
    \end{tikzpicture} \\
    \subcaption[0.4\textwidth]{A planar embedding of graph $P_{\varphi}$ with blue and red color represent positive and negative edges respectively.} & \subcaption[0.4\textwidth]{The graph $G_{\varphi,p}$ where the blue and red dashed lines denote positive and negative edge gadgets respectively.}
    \end{tabular}
    \caption{(a) The graph $P_{\varphi}$ corresponding to $\varphi = (x_1 \vee x_2 \vee \overline{x}_3) \wedge (\overline{x}_2 \vee x_3 \vee x_4) \wedge (x_1 \vee x_3 \vee \overline{x}_4)$, and (b) The graph $G_{\varphi,p}$ obtained with suitable replacement by clause and edge gadgets. }
    \label{fig:my_label}
\end{figure}
First, we take a planar embedding of $P_{\varphi}$. For each clause vertex $C_j$, we order the three edges incident on $C_j$ as $e_1, e_2$ and $e_3$ in a clockwise manner with respect to the considered planar embedding of $P_{\varphi}$. Next, we replace these clause vertices with their corresponding clause gadgets, and the three edges $e_1,e_2$ and $e_3$ incident on $C_j$ are distributed to the vertices $c_{j1}, c_{j2}$ and $c_{j3}$ respectively, such that the resulting graph embedding is planar. For an edge $e_l = (x_i, C_j)$ in $P_{\varphi}$, if the variable $x_i$ occurs as a positive literal in the clause $C_j$, then we replace $(x_i, c_{jr})$ with a positive edge gadget $[x_i,c_{jr}]$. Otherwise, we replace the edge $(x_i, c_{jr})$ with a negative edge gadget. This completes the construction of $G_{\varphi, p}$. It can be observed that $G_{\varphi, p}$ is a bipartite planar graph.

If $n_{\varphi, p}$ is the number of vertices in $G_{\varphi, p}$ then we show that $n_{\varphi, p} \leq pq_{\varphi}$. Each forbidden gadget has exactly $3p+5$ vertices. This implies that each clause gadget has exactly $3 + 3(3p+5) = 9p+18 \leq 27p$ vertices.  Also, each edge gadget has at most $10(3p+5) + 4 \leq 84p$ vertices. From this it follows that $n_{\varphi, p}$ has at most $27pk + 3k\cdot84p = 279pk < pq_{\varphi}$.

The following lemma plays a vital role in proving the inapproximability result for \texttt{Min-Near-3-Choosability}.

%\begin{lemma}
%For a given planar $3$-SAT formula $\varphi$, the following holds:
% \begin{enumerate}[label=(\alph*)]
%  \item If $\varphi$ is not satisfiable, then $\mathbb{OPT}(G_{\varphi,p}) > p$.
%  \item If $\varphi$ is satisfiable, then $\mathbb{OPT}(G_{\varphi,p}) < q_{\varphi}$.
% \end{enumerate}
%\end{lemma}

\begin{lemma} \label{lem-a}
Let $\varphi$ be a planar 3-CNF formula and let $OPT(G_{\varphi,p})$ be the size of a minimum near 3-choosable decomposition set in $G_{\varphi,p}$. Then the following statements hold.
\begin{enumerate}[label=(\alph*)]
        \item If $\varphi$ is not satisfiable, then $\mathrm{OPT}(G_{\varphi,p}) > p$.
        \item If $\varphi$ is satisfiable, then $\mathrm{OPT}(G_{\varphi,p}) < q_{\varphi}$.
 \end{enumerate}
\end{lemma}
\begin{proof}
    \begin{proofcnt} 
    \proofcount{
    Let us assume that $\varphi$ is not satisfiable. We take a minimum near $3$-choosable deletion set $A$, and define a truth assignment $\tau_A$. The truth value of a variable $x_i$ is true if and only if the corresponding variable-vertex $x_i$ lies in $A$. The fact that $\varphi$ is unsatisfiable implies the existence of an unsatisfied clause $C_j$ under $\tau_A$. For each edge $(x_i,c_{jr})$ in $P_{\varphi} $ the variable vertex $x_i$ in $A$ propagates the choice of $c_{jr}$ in the edge gadgets depending on the polarity of the edge $(x_i,c_{jr})$. Consequently, none of the vertices $c_{j1}, c_{j2}, c_{j3}$ in the clause gadget $C_j$ belong to $A$. As the clause gadget $C_j$ is not $2$-choosable, $A$ must include at least one of $w_{j1}, w_{j2}, w_{j3}$. Since these three vertices are the root of forbidden gadgets, we can conclude that $|A| > p$.}

    \proofcount{
     Let $\tau$ be a truth assignment that satisfies the formula $\varphi$. From $\tau$ we construct the set $A$ by including the vertices as given below.
     \begin{itemize}
         \item Include all the core vertices of forbidden gadgets in $G_{\varphi,p}$.

         \item For each positive edge gadget $[x_i,c_{jr}]$, if $\tau(x_i)$ is True, then we add the blue vertices in $A$. Otherwise, we add the red vertices in $A$.

         \item For each negative edge gadget $[x_i,c_{jr}]$, we add the red vertex in $A$ if $\tau(x_i)$ is true; otherwise we add the blue vertex in $A$. 
    \end{itemize}
        Let $C_j$ be a clause in $\varphi$, and let $x_i$ be a variable corresponding to some literal in $C_j$ that evaluates to true under $\tau$. Add the blue vertices from the edge gadget $[x_i,c_{jr}]$ along with core vertices from the clause gadget $C_j$ and edge gadget $[x_i,c_{jr}]$ to $A$. This process can be repeated for each clause in $\varphi$. The satisfiability of $\varphi$ and construction of $G_{\varphi,p}$ implies that the graph obtained by deleting the vertices in $A$ from $G_{\varphi,p}$ is $2$-choosable. Notably, $A$ is an independent set, and a minimum near $3$-choosable deletion set of $G_{\varphi,p}$.
       To determine the size of the set $A$, we consider the contribution of each clause and edge gadget to $A$. Each clause gadget contributes at most $3$ core vertices and $3$ non-root vertices to $A$, while each edge gadget contributes at most $10$ core vertices and $2$ non-root vertex (since the other one is already counted in the clause gadget). Since each clause gadget in $G_{\varphi,p}$ has $3$ edge gadgets attached to it, we have $|A| \leq 6k + 3k \cdot 12 = 42k$, which is less than $q_\varphi$. 
     }  
  \end{proofcnt} 
  \end{proof} 

%From Lemma \ref{lem-a} we have the following conclusion. If $OPT(G_{\varphi,p}) < p$ then $\varphi$ has a satisfying truth assignment. Also, if $OPT(G_{\varphi,p}) > q_{\varphi}$ then $\varphi$ has no satisfying truth assignment.

Based on Lemma \ref{lem-a}, we can draw the following conclusions. If $OPT(G_{\varphi,p}) < p$, then $\varphi$ has at least one satisfying truth assignment. On the other hand, if $OPT(G_{\varphi,p}) > q_{\varphi}$, then $\varphi$ has no satisfying truth assignment.

Assume that \texttt{Min-Near-3-Choosability} is approximable within a factor of 
$|V|^{1-\epsilon}$, for some fixed $\epsilon \in (0, 1)$, where $|V|$ is the number of vertices in the input graph $G$.
%Assume that \texttt{Min-Near-3-Choosability} has a polynomial-time $|V|^{1-\epsilon}$-factor algorithm, for some fixed $\epsilon \in (0, 1)$, where $|V|$ is the number of vertices in the input graph $G$. 
Let $APX(G_{\varphi,p})$ be the size of a solution of \texttt{Min-Near-3-Choosability} for the instance $G_{\varphi,p}$ computed by this algorithm. We assume that $n_{\varphi,p}$ is the number of vertices in $G_{\varphi,p}$. Thus, we have
$$OPT(G_{\varphi,p}) \leq APX(G_{\varphi,p}) \leq n_{\varphi,p}^{1-\epsilon}OPT(G_{\varphi,p}).$$
If $APX(G_{\varphi,p}) < p$, then $OPT(G_{\varphi,p}) < p$. By applying Lemma \ref{lem-a} and its conclusion, it follows that $\varphi$ is satisfiable.

If $APX(G_{\varphi,p}) \geq p$, then $n_{\varphi,p}^{1-\epsilon} OPT(G_{\varphi,p}) \geq p$. Hence,
$$OPT(G_{\varphi,p}) \geq \frac{p}{n_{\varphi,p}^{1-\epsilon}} > \frac{p}{(pq_{\varphi})^{1-\epsilon}} = \frac{p^{\epsilon}}{q_{\varphi}^{1-\epsilon}} \geq \frac{(q_{\varphi}^{(2-\epsilon)/{\epsilon}})^{\epsilon}}{q_{\varphi}^{1-\epsilon}} = q_{\varphi}. $$
Thus, according to the conclusion of Lemma \ref{lem-a}, $\varphi$ has no satisfying truth assignment.
Therefore, we can decide whether $\varphi$ is satisfiable or not in polynomial time, implying that \texttt{Planar-3-SAT} is in $\mathbb{P}$. This completes the proof of Theorem \ref{thm-inappx}.  
\end{proof}

\section{Minimum 2-choosable deletion problem}
In this section, we consider the complexity of Minimum 2-Choosable Deletion problem (\texttt{Min-2-Choosable-Del}) which is a relaxation of \texttt{Min-Near-3-Choosability}. A subset $A \subseteq V$ is called a 2-choosable deletion set in $G$ if $G[V_G \setminus A]$ is 2-choosable.
Given a graph $G$, in \texttt{Min-2-Choosable-Del} it is required to find a 2-choosable deletion set $A$ of minimum size. It is easy to observe that \texttt{Min-2-Choosable-Del} is $\mathbb{NP}$-complete. This reduction is exactly the same as the reduction from minimum vertex cover to minimum feedback vertex set. In this reduction, from a graph $G$ (an instance of minimum vertex cover) a graph $G'$ (an instance of minimum feedback vertex set) is constructed in polynomial time. First, we make a copy of $G$, and then for each edge $e=(u, v) \in E_G$, we introduce a new vertex $v_e$ and a pair of edges $(u, v_e)$ and $(v, v_e)$. It is easy to observe that in this construction, each edge $e=(u, v)$ is replaced with a triangle consisting of vertices $u, v$ and $v_e$. 
It can be proved that $S$ is a vertex cover in $G$ if and only if $S$ is a feedback vertex set in $G'$.
Since an odd cycle is not 2-choosable, this reduction can be seen as a polynomial time reduction from minimum vertex cover to \texttt{Min-2-Choosable-Del}.

We design a $O(\log n)$-factor approximation algorithm for \texttt{Min-2-Choosable-Del} which is similar to the $O(\log n)$-factor approximation algorithm for minimum feedback vertex set \cite{bar1998approximation}. 
From the proof of Theorem (A. L. Rubin) in \cite{erdos} it follows that $G$ is not 2-choosable if $G$ contains a subgraph that belongs to one of the following five types of graphs. (a) An odd cycle, (b) two vertex disjoint even cycles connected by a path, (c) two even cycles having exactly one common
vertex, (d) $\theta_{a, b, c}$ with $a \neq 2$ and $b \neq 2$, and (e) $\theta_{2,2,2,t}$ with $t \geq 2$. Here $\theta_{2,2,2,t}$ is the graph consisting of two distinguished vertices $u$ and $v$ together with 4 vertex disjoint paths of length 2, 2, 2, and $t$ sharing common endpoints $u$ and $v$. Since graphs in each of these five types contain a cycle, we design an algorithm for \texttt{Min-2-Choosable-Del} by selecting cycles of small length to destroy these types of subgraphs.

\begin{algorithm2e}
\SetAlgoLined
\KwIn{$G=(V, E)$}
\KwOut{A graph $H$ without having vertices of degree 1 and long induced paths}
$H = G$\;
\While{there exists a vertex $v$ of degree 1 in $H$}{
  Delete $v$ from $H$ and the edge incident on it\;
}
\If{$H$ has a maximal induced path $P(u_1, u_k)=\langle u_1, u_2, \cdots, u_r \rangle $ with both $u_1$ and $u_r$ have degree 2, $r \geq 2$}{
    \tcc{In $P(u_1, u_r)$, $u_1$ and $u_r$ are distinct.}
    Delete all the vertices in $P(u_1, u_r)$ and the edges incident on them\;
    Introduce a new vertex $s$ with $c(s)= c(u_1) + \cdots + c(u_r)$\;
    Introduce the edges $(w, s)$ and $(s, z)$, where $w$ and $z$ are the neighbours of $u_1$ and $u_r$ not in this path, respectively\;
}
Return $H$\;
\caption{Preprocessing($G$)}
\label{algoPrePro}
\end{algorithm2e}

Based on Theorem \ref{thm:2choosablethm}, we define a preprocessing algorithm (see Algorithm \ref{algoPrePro}). Here, we will assign a count $c(v)=1$ to each vertex  $v \in V_G$. In Algorithm \ref{algoPrePro}, we replace an induced path $P(u_1, \cdots, u_r) = \langle u_1, u_2, \cdots, u_r \rangle$ (with $r \geq 2$) by a new vertex $s$ and assign its count as the sum of counts of the vertices in $P(u_1, u_r)$. 
%Based on this notion, we define the length of a cycle/path is the sum of counts of the vertices of cycle/path.
If $G$ is a cycle, then preprocessing algorithm chooses a path $P(u_1, u_k)$ consisting of all the vertices of this cycle except one vertex. This preprocessing algorithm converts this cycle into a multigraph with two vertices and two parallel edges. We define $\mathcal{ C}'$ as a set of three types of graphs defined as follows. (a) $K_1$ with $c(p) \geq 1$, where $p$ is the only vertex in $K_1$. (b) A multigraph $T$ with $V_T= \{p, q\}$, two parallel edges among $p, q$ and $c(p) + c(q)$ is an even integer. (c) $K_{2, 3}$ with $c(v) = 1$, for each vertex $v$ except one vertex $u$ with $c(u)$ is an odd integer and $d(u)=2$. It can be observed that $G$ is a connected 2-choosable graph if and only if Preprocessing($G$) belongs to $\mathcal{ C}'$.

Let $G$ be a connected graph that is not 2-choosable and $H$ be the graph obtained from $G$ by using Preprocessing algorithm. If $d_H(v)=2$ and it has two distinct neighbors $p$ and $q$ then both $p$ and $q$ have degrees at least 3. By using this property of $H$ and the results in \cite{Os2022OnTM,bar1998approximation}, a breadth-first-search based algorithm can compute a cycle $T$ with at most $O(\log |V_H|)$ vertices.

\begin{algorithm2e}
\SetAlgoLined
\KwIn{$G=(V, E)$}
\KwOut{A subset $A$ of $V$ such that $G[V\setminus A]$ is 2-choosable}
$H = G$ and $A=\emptyset$\;
$H = Preprocessing(G)$\;
Delete all the components in $H$ that are in $\mathcal{ C}'$\;
\While{$H$ has at least one vertex}{
  Compute a cycle $T$ of length at most $O(\log n)$ in $H$\;
  $A = A \cup V_T$\;
  Delete all the vertices in $V_T$ from $H$ and the edges incident on them\;
  $H = Preprocessing(H)$\;
  Delete all the components in $H$ that are in $\mathcal {C}'$\;
}
Return $A$\;
\caption{2-Choosable-Deletion($G$)}
\label{algo2-choosable-del}
\end{algorithm2e}

In general, the set $A$ computed by Algorithm \ref{algo2-choosable-del} may not be a subset of $V_G$. If $v \in A$ is a vertex with $c(v) > 1$ then we will replace $v$ with the first vertex $u$ of the path represented by $v$. After this modification, $A$ becomes a subset of $V_G$ and a $2$-choosable deletion set in $G$.

\begin{theorem}
 %Algorithm \ref{algo2-choosable-del} is an $O(\log |V_G|)$ factor approximation algorithm for \texttt{Min-2-Choosable-Del}.
 \texttt{Min-2-Choosable-Del} can be approximated within a factor of $O(\log |V_G|)$.
\end{theorem}
\begin{proof}
In the Algorithm \ref{algo2-choosable-del}, let $T$ be a cycle in $H'$ found by the algorithm in some iteration. After deletion of the vertices of $T$, let $J$ be the set of vertices removed from $H'$ by the preprocessing step and the 2-choosable components deletion step in the same iteration. We say that the vertices in $J \cup T$ are eliminated by $T$.

Let $T_1, T_2, \cdots, T_r$ be the cycles computed by the Algorithm \ref{algo2-choosable-del}. Now, it is easy to observe that every vertex in a 2-choosable deletion set of $G$ must be eliminated by at least one of these $r$ cycles. Also, these $r$ cycles are vertex disjoint. Since the algorithm includes all the vertices of the computed cycle $T$ into $A$ at each iteration, the error incorporated is at most $|V_T| \leq O(\log |V_G|)$.  
\end{proof}

\section{Conclusion}
%The problems on near $k$-choosablity are of particular interest for non-$k$-choosable graphs since they are the immediate refinement of $\{k\}$-choosability. The characterization of $2$-choosable graphs in \cite{erdos} makes it interesting to study near $3$-choosability of graphs. It would be interesting to determine the decidabiliy of \texttt{Near-3-Choosability} for other families of graphs, especially for $3$-colorable graphs. Also, Question \ref{ques: l-k_choosable} opens a broad range of problems, answers to which will shed new light on complexity of choosability of graphs. It would be interesting to know if there exists a constant factor algorithm for \texttt{MIN-2-Choosable-Del}, which we believe is not true. It would also be intriguing to explore the existence of a constant factor algorithm for \texttt{MIN-2-Choosable-Del}, as we suspect that such an algorithm does not exist.

The study of near $k$-choosability problems is particularly significant for graphs which are not $k$-choosable, as they offer an immediate refinement of ${k}$-choosability. The interest in studying near $3$-choosability stems from the characterization of $2$-choosable graphs presented in \cite{erdos}. It would be worthwhile to investigate the decidability of \texttt{Near-3-Choosability} for various families of graphs, especially $3$-colorable graphs. Additionally, Question \ref{ques: l-k_choosable} poses a wide range of problems, the solutions to which can provide new insights into the complexity of graph choosability. It would be interesting to know if there exists a constant factor algorithm for \texttt{MIN-2-Choosable-Del}, which we believe is not true.

\end{document}